\documentclass{article}
\usepackage{amsmath,amssymb,
color,theorem}

\newcommand{\hide}[1]{}
 \newcommand{\tmop}[1]{\ensuremath{\operatorname{#1}}}
 \newcommand{\tmstrong}[1]{\textbf{#1}}
 \newcommand{\tmtextbf}[1]{{\bfseries{#1}}}
 \newcommand{\tmtextrm}[1]{{\rmfamily{#1}}}
 \newcommand{\tmtextup}[1]{{\upshape{#1}}}
\newcommand{\mathup}[1]{{\rm{#1}}}

\newtheorem{defn}{Definition}

\newtheorem{lemma}[defn]{Lemma}
\newtheorem{proposition}[defn]{Proposition}
\newtheorem{theorem}[defn]{Theorem}

\newtheorem{remark}[defn]{Remark}
\newtheorem{notation}[defn]{Notation}
\newtheorem{example}[defn]{Example}

           {\vspace{3.3mm}
           \noindent{\bf #1}\it}%
           {\vspace{3.3mm}}

\newenvironment{proof}[1]{
  \trivlist \item[\hskip \labelsep{\it #1}]}{\hfill\mbox{$\square$}
  \endtrivlist}

\begin{document}

\title{Zero-nonzero and real-nonreal sign determination}
\author{Daniel Perrucci$^{\flat, \natural}$\footnote{Partially supported by the following grants: 
PIP 099/11 CONICET and UBACYT 20020090100069 (2010/2012).}  , Marie-Fran\c{c}oise Roy$^{\diamondsuit}$,
\\[3mm]
{\small ${\flat}$ Departamento de Matem\'atica, FCEN, Universidad de Buenos Aires, Ciudad Universitaria, 
1428 Buenos Aires, Argentina}\\ 
{\small ${\natural}$ CONICET, Argentina}\\ 
{\small ${\diamondsuit}$ IRMAR (UMR CNRS 6625), Universit\'e de Rennes 1, 
Campus de Beaulieu, 35042 Rennes, cedex,  France}
}

\maketitle

\begin{abstract}
We consider first the zero-nonzero determination problem, which consists in determining the list
of zero-nonzero conditions realized by a finite list of polynomials on a finite set $Z \subset
\mathup{C}^k$ with $\mathup{C}$ an algebraic closed field. We describe an algorithm 
to solve the zero-nonzero determination problem and we perform its bit complexity analysis. 
This algorithm, which is in many ways an adaptation of the
methods used to solve the more classical sign determination problem, presents also  new 
ideas which can  be used to improve 
sign determination. Then, we consider the real-nonreal sign determination problem, 
which
deals
with both the sign determination and the zero-nonzero determination
problem. We describe an algorithm to 
solve 
the real-nonreal sign determination problem, 
we perform its bit complexity analysis and we discuss this problem in a parametric context. 
\end{abstract}

\section{Introduction}

Let $\mathup{L}$ be a field and $\mathup{C}$ an algebraically closed extension of $\mathup{L}$. 
Consider a finite set $Z \subset \mathup{C}^k$
and a finite list of polynomials $\mathcal{P} = P_1, \ldots, P_s$  in
$\mathup{L}[X_1, \ldots, X_k]$;  
the zero-nonzero determination problem is the problem of computing 
the
zero-nonzero conditions of $\mathcal{P}$ which are
realized on $Z$. In order to better explain this, we introduce some notation and definitions.

For $a \in \mathup{C}$, its {\tmstrong{invertibility}} is defined as follows:
\[ \begin{cases}
     \tmop{inv} (a) = 0 & \tmop{if} a = 0,\\
     \tmop{inv} (a) = 1^{} & \tmop{if} a \ne 0.\\
   \end{cases} 
\]
Let $I \subset \left\{ 1, \ldots, s
\right\}$. Given a {\tmstrong{zero-nonzero condition}} $\sigma \in \{0, 1 \}^I$, the 
\tmtextbf{realization of $\sigma$ on $Z$} is
\[
\tmop{Reali}
(\sigma, Z) = \Big\{x \in Z \mid \bigwedge_{i
\in I} \tmop{inv} (P_i (x)) = \sigma(i) \Big\}  
\]
and we denote by ${c
(\sigma, Z})$ the cardinal of $\tmop{Reali}
(\sigma, Z)$. 
We write 
$\tmop{Feas}
( \mathcal{P},
Z)$ for the list of $\sigma \in \{0, 1\}^{\left\{ 1,
\ldots, s \right\}}$ such that ${c
(\sigma, Z})$ is not zero, and
$c
( \mathcal{P}, Z)$ for the corresponding list of cardinals.
The \tmtextbf{zero-nonzero determination problem} is to determine $\tmop{Feas}
( \mathcal{P},
Z_{})$ and $c
( \mathcal{P}, Z)$.

Typically, the set $Z$ is not known explicitelly, 
but given as the 
complex
zero set of a 
polynomial system; therefore, to solve the zero-nonzero determination problem it is not possible
to simply evaluate the polynomials in  $\mathcal{P}$ separately at each point of $Z$, and a more
clever strategy is needed.

The zero-nonzero determination problem is
anologous to
the more 
classical sign 
determination problem, which we recall now. 
Let $\mathup{K}$ be an ordered field and $\mathup{R}$  a real closed extension of $\mathup{K}$. 
For $a \in \mathup{R}$, its {\tmstrong{sign}} is defined as follows:
\[ \begin{cases}
   \tmop{sign} (a) = 0 & \tmop{if} a = 0,\\
     \tmop{sign} (a) = 1^{} & \tmop{if} a > 0,\\
     \tmop{sign} (a) = -1 & \tmop{if} a < 0.\\
   \end{cases} \]

Consider a finite set $W \subset \mathup{R}^k$
and a finite list of polynomials $\mathcal{P} = P_1, \ldots, P_s$  in
$\mathup{K}[X_1, \ldots, X_k]$. 
Given a {\tmstrong{sign condition}} $\tau \in \{0, 1, - 1\}^{\{1, \dots, s\}}$, the
\tmtextbf{realization of $\tau$ on $W$} is
 \[ 
 \tmop{Reali}
 _{\rm sign} 
 (\tau, W) = \Big\{ x \in W \mid
     \bigwedge_{i = 1, \dots, s} \tmop{sign}(P_i (x)) = \tau(i)\Big\} 
 \]
and we denote by $c
_{\rm sign}
(\tau, W)$ the cardinal of $\tmop{Reali}
_{\rm sign}
(\tau, W)$. 
We write 
$\tmop{Feas}_{\rm sign} ( \mathcal{P},
W)$ for the list of $\tau \in \{0, 1, -1\}^{\left\{ 1,
\ldots, s \right\}}$ such that $c
_{\rm sign}
(\tau, W)$ is not zero, and
$c
_{\rm sign}
( \mathcal{P}, W)$ for the corresponding list of cardinals.
The \tmtextbf{sign determination problem} is to determine $\tmop{Feas}
_{\rm sign} 
( \mathcal{P},
W)$ and $c
_{\rm sign}
( \mathcal{P}, W)$. Once again, 
the set $W$ is typically not known explicitelly, 
but given as the
real
zero set of a 
polynomial system.

Let $Q \in \mathup{K}[X_1, \dots, X_k]$, the {\tmstrong{Tarski-query}} of $Q$ for
$W$  
is 
\[ 
\tmop{TaQu} (Q, W) = \sum_{x \in W}
   \tmop{sign} (Q (x)) = \tmop{card} \left( \left\{ x \in W \mid
   Q \left( x \right) > 0 \right\} \right) - \tmop{card} \left( \left\{ x \in
   W \mid Q \left( x \right) < 0 \right\} \right). 
\]
Tarski-queries play a leading role in the most efficient algortihms to solve the sign determination problem 
(\cite{RS, Canny, BPR}). In fact, these algorithms consist in 
computing
a relevant list of  Tarski-queries and solving 
linear systems 
with integer coefficients  
having a specific structure. 
Suppose that the polynomial system defining $W$ 
is in $\mathup{K}[X_1, \dots, X_k]$. Then, in the mentioned algorithms for sign 
determination there are three different kind of operations:
\begin{itemize}
\item sign comparisons,
\item operations in $\mathbb{Q}$, which appear in the linear  solving steps,
\item operations in $\mathup{K}$, which appear in the Tarski-query computations. 
\end{itemize}
Regarding the complexity analysis, in \cite[Chapter 10]{BPR} 
the Tarski-query computation is 
considered as a blackbox. 
Indeed, there exist many well-known methods to compute them, and 
depending on the setting, the application of one method or another is convenient. 
However, when asymptotically fast methods for computing the Tarski-query are used
in the univariate case,  the cost of solving the linear system dominates the overall complexity
as already noticed in  \cite[Section 3.3]{Canny}.
In \cite{Per}, a 
method for solving the 
specific linear systems arising in the sign determination
algorithm is given, leading to a complexity improvement
(see \cite[Corollary 2]{Per}).

The real-nonreal sign determination problem is 
a compressed way of dealing with both the sign determination and the zero-nonzero determination
problem as we explain now. 
Let $\mathup{K}$ be an ordered field, $\mathup{R}$ a real closed extension of $\mathup{K}$
and $\mathup{C} = \mathup{R}[i]$. Consider a finite set $Z \subset \mathup{C}^k$, 
one more time, typically given as the complex zero set of polynomial system,  
and a finite
list of polynomials ${\cal P} = P_1, \dots, P_s$ in $\mathup{K}[X_1, \dots, X_k]$. 
We define
\begin{eqnarray*}
  Z_{\mathup{R}} & = & Z \cap \mathup{R}^k, \\
  Z_{\mathup{C} \setminus \mathup{R}} & = & Z \cap \left( \mathup{C}^k
  \setminus \mathup{R}^k \right).
\end{eqnarray*}
The \tmstrong{real-nonreal sign determination problem} is
to determine 
$\tmop{Feas}_{\rm sign} ( \mathcal{P}, Z_{\mathup{R}})$, 
$\tmop{Feas} ( \mathcal{P}, Z_{\mathup{C} \setminus \mathup{R}})$,  
$c_{\rm sign} ( \mathcal{P},$ $Z_{\mathup{R}})$ and
$c( \mathcal{P},
Z_{\mathup{C} \setminus \mathup{R}})$.
By solving the real-nonreal sign determination problem we obtain a complete
description of the sign and invertibility of the polynomials in ${\cal P}$ on $Z$. 

This paper serves 
several purposes.
First, we give an algorithm 
for zero-nonzero determination. This algorithm is 
based on the same principles that sign determination, 
the role of Tarski-queries 
being played by  
invertibility-queries, which we introduce now.
Let $Q \in \mathup{L}[
X_1, \ldots, X_k]$, the {\tmstrong{invertibility-query}} of $Q$ for $Z$
is
\[ \tmop{Qu} (Q, Z) = \sum_{x \in Z} \tmop{inv} (Q (x)) = \tmop{card} \left(
   \left\{ x \in Z \mid Q \left( x \right) \ne 0 \right\} \right) . \]
Then we 
perform
a complexity analysis of this algorithm as follows: we 
consider the 
invertibility-queries as a blackbox, 
and estimate the \emph{bit complexity}
of the zero-nonzero comparisons and the operations in $\mathbb{Z}$.  
Note that the algorithm we present here is 
not a straightforward adaptation of the known algorithms for sign determination.
In fact, in order to obtain a good bit complexity bound,  we introduce
some new definitions (i.e. the compression and a matrix summarizing the useful information) which
can be used in turn to improve  the known algorithms for sign determination 
(see \cite[Chapter 10]{BPRcurrent}). 

Then, combining sign and zero-nonzero determination, we give an algorithm for real-nonreal 
sign determination and perform a complexity analysis in a similar way.  
A final
purpose of this paper is to discuss real-nonreal 
sign determination
in a parametric context.

This paper is organized as follows. 
In Section \ref{sec:comp} we 
give 
an algorithm for zero-nonzero determination
and 
perform its
bit complexity analysis considering the invertibility-query as a blackbox.
In Section \ref{sec:nonrealqueries} we 
give 
an algorithm for real-nonreal sign determination
and perform
its bit
complexity analysis in a similar way.
In Section \ref{sec:queries} we explain the various existing methods for
computing the Tarski-queries and invertibility-queries
in the univariate and
multivariate case
and we deduce the bit complexity of zero-nonzero 
and real-nonreal sign determination in the univariate case. 
Finally, in Section \ref{sec:parameters} we discuss real-nonreal sign determination
in a parametric context.

\section{The zero-nonzero determination problem}\label{sec:comp}

We recall that $\mathup{L}$ is a field and $\mathup{C}$ an algebraically closed extension of $\mathup{L}$, 
$Z \subset \mathup{C}^k$  a finite set, 
$\mathcal{P} = P_1, \ldots, P_s$ a list of polynomials in
$\mathup{L}[X_1, \ldots, X_k]$, $I$ a subset of $\{ 1, \ldots, s
\}$ and 
$\sigma \in \{0, 1 \}^I$ a 
zero-nonzero condition.

\subsection{Definitions and properties}

Given $J \subset I \subset \{ 1, \ldots, s \}$ and $\sigma \in \{0, 1\}^{I}$, 
we write  $\mathcal{P}^J$
for~$\prod_{i \in J}^{} P_i^{}$ and $\sigma^J$ for $\prod_{i \in J} \sigma
(i)^{}$. 
Note that when $\tmop{Reali} (\sigma, Z) \ne
\emptyset$, the value of $\tmop{inv} ( \mathcal{P}^J (x)
)$ is fixed as $x$ varies in $\tmop{Reali} (\sigma, Z)$ and is equal to
$\sigma^J$. By convention 
$\mathcal{P}^{\emptyset} = 1$, 
$\sigma^{\emptyset} = 1$, 
and 
$\{  0, 1 \}^{\emptyset} = \{
\emptyset \}$.

For a fixed $I \subset \left\{ 1,
\ldots, s \right\}$, 
we consider the lexicographical order on  $\{0, 1\}^{I}$ (with $0 < 1$), 
identifying a zero-nonzero condtion in $\{0, 1\}^{I}$ with a 
bit string of length 
$\tmop{card}(I)$. Throughout this paper, all the lists 
of zero-nonzero conditions
we consider are ordered 
and have no repetitions. 
By the union of 
two disjoint lists we mean their ordered union.

Given $I \subset \{ 1, \dots, s \}$ and 
a list $A = \left[ I_1, \dots, I_m \right]$ of subsets of $I$, 
we define $\tmop{Qu} (
\mathcal{P}^A, Z)$ as the vector with coordinates $\tmop{Qu} (
\mathcal{P}^{I_1}, Z), \ldots, \tmop{Qu} ( \mathcal{P}^{I_m}, Z)$.
Also, given a list 
$\Sigma = [\sigma_1, \dots, \sigma_n]$ of zero-nonzero condiditions in 
$\{0, 1\}^{I}$, we define $c(\Sigma, Z)$ as the
vector with coordinates $c(\sigma_1, Z), \dots, c(\sigma_n, Z)$.
The \tmtextbf{matrix of $A$ on $\Sigma$} 
is the $m \times n$ matrix $\tmop{Mat} (A, \Sigma)$
whose {$i, j$-th} entry is $\sigma_j^{I_i}$ for $i = 1,
\ldots, m$, $j = 1, \ldots, n$.
By convention $\tmop{Mat} (\emptyset, \emptyset)=\emptyset$ 
(i.e. the empty matrix with $0$ rows and columns) and is invertible.

Let $\mathcal{P}_I=\{P_i\mid i\in I\}$. With this notation, we have the following:
\begin{proposition}
\label{pro:matrix of signs} Let $I \subset \{1, \dots, s\}$ and 
 $\Sigma  \subset \{0, 1\}^I$ 
with $\tmop{Feas}(
\mathcal{P}_I
,Z) \subset \Sigma$.
Then
  \[ \tmop{Mat} (A, \Sigma) \cdot c (\Sigma, Z) = \tmop{Qu} ( \mathcal{P}_I^A,
     Z) . \]
\end{proposition}

\begin{proof}{Proof:}
  The claim is easy if $Z$ has a single element.  Indeed, suppose $Z = \{ x \}$, 
and $\sigma_j$  is the zero-nonzero condition in $\Sigma$ satisfied by
  $\mathcal{P}_I$ at $x$. The coordinates of $c (\Sigma, \{ x \})$
  are $0$ except at place $j$, where it is $1$, so that $\tmop{Mat} (A,
  \Sigma) \cdot c (\Sigma, \{ x \})$ is the $j$-th column of
  $\tmop{Mat} (A, \Sigma)$. Its $i$-th coordinate is $\sigma_j^{I_i} =
  \tmop{inv} ( \mathcal{P}_I^{I_i} ( x ) ) = \tmop{Qu}
  (\mathcal{P}_I^{I_i}, \{ x \} )$. 

In the general case, the claim follows by
  linearity since $c (\Sigma, Z) = \sum_{x \in Z} c( \Sigma, \{ x
  \})$ and $\tmop{Qu} ( \mathcal{P}_I^A, Z) = \sum_{x \in Z}
  \tmop{Qu} ( \mathcal{P}_I^A, \{ x \})$.
\end{proof}

\begin{example}
When $I = \{i\} \subset \{1, \dots, s\}$, $A = [\emptyset, \{ i \}]$, and
$\Sigma = [0, 1]$, the conclusion of Proposition \ref{pro:matrix of
signs} is
$$ \left(\begin{array}{cc}
    1 & 1\\
    0 & 1
  \end{array}\right) \cdot 
  \left(\begin{array}{c}
    c (P_i = 0, Z)\\
    c (P_i \ne 0, Z)
  \end{array}\right) = \left(\begin{array}{c}
    \tmop{Qu} (1, Z)\\
    \tmop{Qu} (P_i, Z)
  \end{array}\right).
$$
\end{example}

It follows from Proposition \ref{pro:matrix of signs} that 
if $\Sigma \subset \{0,1\}^{\{1, \dots, s\}}$ 
contains
$\tmop{Feas} \left( \mathcal{P}, Z \right)$
and the matrix $\tmop{Mat} (A, \Sigma)$ is invertible, we can
compute $c (\Sigma, Z)$ from~$\tmop{Qu} ( \mathcal{P}^A, Z)$. 
Therefore, if we take $A$ as the list of the $2^s$ subsets of $\{1, \dots, s\}$ 
and $\Sigma$ as the list of the $2^s$ zero-nonzero conditions in $\{0, 1\}^{\{1, \dots, s\}}$,
since
it is easy to check that
$\tmop{Mat} (A, \Sigma)$ is invertible, we have a  naive
method for zero-nonzero determination: we solve the  
$2^s \times 2^s$ linear system from Proposition \ref{pro:matrix of signs} and 
we discard the zero-nonzero conditions $\sigma$ such that $c(\sigma, Z) = 0$. 
This naive method involve
an exponential number of invertibility queries and solving a linear system of
exponential 
size. We want an algorithm with a better complexity bound.

The key fact is to take into account that 
the number of realizable zero-nonzero conditions does not
exceed $\tmop{card} (Z)$. We are going to consider one by one the polynomials $P_1, \dots, P_s$
in the list $\mathcal{P}$ and to compute at  step $i$ the realizable zero-nonzero conditions for 
the list  $\mathcal{P}_i := P_1, \dots, P_i$, determining the nonempty
zero-nonzero conditions inductively and getting rid of the empty ones at each step. In this way, the size of
the data we manipulate is well controlled.

We need some preliminary definitions and results.

\begin{defn}\label{defn_adapted_list}
Let $I \subset \left\{ 1, \ldots, s \right\}$ and $\Sigma \subset \{0, 1\}^I$.
A set $A$ of subsets of $I$ is
\tmtextbf{adapted to zero-nonzero determination on $\Sigma$}
if the matrix of $A$ on $\Sigma$ is
invertible.
\end{defn}

Note that the set $A$ has to be ordered into a list so that the
matrix of $A$ on $\Sigma$ is unambiguous, but the choice of this ordering 
 does not change the fact that the matrix is
invertible.

\begin{example}\label{example_adapted} \begin{enumerate}
\item Consider $I = \emptyset$ and $\Sigma=[\emptyset]$, then
$ A=\{\emptyset\}$ is adapted to  zero-nonzero determination  on $\Sigma$
 since $\tmop{Mat} (A, \Sigma)= (\, 1 \, )$
 is  invertible.

\item Consider $I = \{i\} \subset \{1, \dots, s\}$, then: 
  \begin{itemize}
\item if $\Sigma = [0, 1]$, $A = \{\emptyset, \{i\}\}$ 
is adapted to zero-nonzero determination  on $\Sigma$, since
    \[ \tmop{Mat} (A, \Sigma)= \left(\begin{array}{cc}
         1 & 1\\
         0 & 1
       \end{array}\right)\]
       is invertible,
           \item if $\Sigma = [0]$ or $\Sigma = [1]$,  $A = \{\emptyset\}$  
is adapted to zero-nonzero determination on $\Sigma$,
 since $\tmop{Mat} (A, \Sigma)= (\, 1 \, )$
 is  invertible.
    \end{itemize}
\end{enumerate}
\end{example}

Let $I \subset \{ 1, \ldots, s \}$. Our aim is to describe a method for
determining for each $\Sigma \subset \{0, 1\}^I$,
a 
set
$A$ of subsets of $I$ adapted to zero-nonzero determination on
$\Sigma$.
First, we introduce some more definitions and notation.

\begin{defn}
 If $J \subset I \subset \{ 1, \ldots, s \}$, $\sigma
  \in \{ 0, 1 \}^J$ is the {\tmstrong{restriction}} of {$\tau \in
  \{ 0, 1\}^I$} if $\sigma (j) = \tau (j)$ for every $j \in J$; we also say
  that $\tau$ is an  {\tmstrong{extension}} of $\sigma$. 
If $\Sigma \subset
  \{ 0, 1 \}^I$,
  we denote by $\Sigma_J \subset \{0, 1\}^J$
the list of restrictions of elements of $\Sigma$ to $J$. 
\end{defn}

\begin{notation}
\label{addindex} 
For a 
set
$A$ of subsets of $J\subset \{1,\ldots,s\}$ and  $j\in \{1,\ldots,s\}$ 
 bigger than $\tmop{max}(J)$, we denote by 
$(A, j)$ the 
set
of subsets  of $J\cup\{j\}$
obtained by adding $j$ to all
the elements of $A$. 
\end{notation}

We are now ready to construct a 
set
of subsets adapted to sign
determination.

\begin{defn}
\tmtextbf{\tmtextup{[Adapted family]}}\label{def:adap}
  Let  $I \subset \{ 1, \ldots, s \}$ and $\Sigma \subset \{
  0, 1\}^I$.
The {\tmstrong{adapted family}} $\tmop{Ada} \left( \Sigma
\right)$ is defined by induction as follows:
\begin{itemize}
\item If $I = \emptyset$, then, if $\Sigma = \emptyset$, define $\tmop{Ada} \left( \Sigma
\right) = \emptyset$, if $\Sigma = [\emptyset]$, define $\tmop{Ada} \left( \Sigma
\right) = \{\emptyset\}$.

\item  If $I \ne \emptyset$, consider 
       $i=\max ( I )$,
$I' = I \setminus \{i \}$, $\Xi=\Sigma_{I'}$ and
    $\Xi'$ the 
   list of elements of
    $\Xi$ having two different extentions in $\Sigma$. Define
    \[ \tmop{Ada} ( \Sigma) = \tmop{Ada} (\Xi) \cup (
       \tmop{Ada} ( \Xi' ), i ). \]
     \end{itemize}
\end{defn}

From the previous definition
it is easy to prove that for  
$I \subset \{ 1, \ldots, s \}$ and $\Sigma \subset \{
  0, 1\}^I$, 
 $\tmop{card}(\tmop{Ada}(\Sigma)) = 
\tmop{card}(\Sigma)$. Before proving that the adapted family $\tmop{Ada}(\Sigma)$ 
we defined is adapted to zero-nonzero determination on $\Sigma$, 
we prove some auxiliary results.

\begin{lemma}  \label{includecompress} 
Let $I \subset \{1, \dots, s\}$ and $\Sigma_1 \subset \Sigma \subset \left\{ 0, 1
  \right\}^I$, 
    then $\tmop{Ada} (\Sigma_1) \subset \tmop{Ada} \left( \Sigma
  \right)$.
\end{lemma}

\begin{proof}{Proof:} We prove the claim by induction on  $\tmop{card}(I)$.
If $I = \emptyset$, the claim is true. 
Suppose now $I \ne \emptyset$.
Following the notation in 
  Definition \ref{def:adap}, and noting that $\Xi_1=(\Sigma_1)_{I'}
  \subset \Xi$ and $\Xi'_1 \subset$ $\Xi'$, 
  where
     $\Xi'_1 \subset \{ 0,1\}^{I'}$ is the list of elements of
    $\Xi_1$ having two different extentions in $\Sigma_1$,  the claim follows using twice
  the induction hypothesis.
\end{proof}

The following result will be useful for the complexity analysis.

\begin{proposition}
  \label{pro:bkr} Let $I \subset \{1, \dots, s\}$ and $\Sigma \subset \{0,1\}^I$. For every $J \in \tmop{Ada} (\Sigma)$,  $\tmop{card} \left(
  J \right) < \tmop{bit} (\tmop{card} \left( \Sigma \right))$.
\end{proposition}

\begin{proof}{Proof: } We will prove by induction on $\tmop{card}(I)$ that for every $J \in 
\tmop{Ada}(\Sigma)$  all the subsets
of $J$ belong to $\tmop{Ada}(\Sigma)$. The proposition follows since 
$2^{\tmop{card}(J)} 
\le \tmop{card}(\tmop{Ada}(\Sigma)) = \tmop{card}(\Sigma)$. 
If $I = \emptyset$, the claim is true. Suppose now $I \ne \emptyset$. 
Taking $i={\rm max}(I)$,
we consider two cases: 
if $i \not \in J$ then
$J \in \tmop{Ada}(\Xi)$, 
if $i \in J$, then $J \setminus \{i\} \in \tmop{Ada}(\Xi')$.
Therefore, using twice 
the induction hypothesis, all the subsets of $J$
not containing $i$ belong to $\tmop{Ada}(\Xi) \subset \tmop{Ada}(\Sigma)$ and
all the subsets of $J$ containing $i$ belong to 
$(\tmop{Ada}(\Xi'),i) \subset \tmop{Ada}(\Sigma)$. 
\end{proof}

\begin{proposition}
  \label{prop:adaptedadapted} Let  $I \subset \left\{ 1, \ldots, s \right\}$
  and $\Sigma \subset \left\{ 1, 0 \right\}^I$.
  The set $\tmop{Ada} ( \Sigma)$ is adapted to zero-nonzero
  determination on $\Sigma$.
\end{proposition}

We adapt the proof in \cite[Proposition 6]{Per}, which provides information about
the inverse of $\tmop{Mat} ( \tmop{Ada} ( \Sigma), \Sigma)$ 
useful later for algorithmic matters.

\begin{proof}{Proof:} First, we define the total
order $\prec$ on the subsets of $\{1, \dots, s\}$, which we use whenever
we have to order a set of subsets into a list. 
Given $J \subset \{1, \dots, s\}$, we associate 
the  natural number $\vert J\vert=\sum_{j\in J} 2^{j-1}$ and then we define for $J_1, J_2 
\subset \{1, \dots, s\}$, 
$J_1 \prec J_2$ if $\vert J_1\vert < \vert J_2\vert$.
Note that $\prec$ extends the partial order of inclusion of subsets.

 We will prove by induction on $\tmop{card}(I)$ that
    $\tmop{Mat} ( \tmop{Ada} ( \Sigma), \Sigma)$
    is invertible. If $I = \emptyset$, then the claim is true.
Suppose now 
$I \ne \emptyset$. We follow the notation in Definition \ref{def:adap}
and 
divide $\Sigma$ in three (possibly empty) sublists:
  \begin{itemize}
    \item $\Sigma_0$ consisting of the elements $\sigma$ of
    $\Sigma$
    such that $\sigma( i ) = 0$ and the restriction of $\sigma$ to
    $I'$ is in $\Xi'$,
        \item $\Sigma_1$ consisting of the elements $\sigma$ of
    $\Sigma$ such that $\sigma( i ) = 1$ and the restriction of
    $\sigma$ to $I'$ is in $\Xi'$,
    
    \item $\Sigma_\star$ consisting of the elements of $\Sigma$ whose restriction to
    $I'$ is in  $\Xi \setminus \Xi'$.
    
\end{itemize}

If $\Xi' = \emptyset$ is empty, 
$\tmop{Mat} ( \tmop{Ada} (\Sigma), \Sigma )=\tmop{Mat} ( \tmop{Ada} (\Xi), \Xi)$ 
is invertible by induction hypothesis.

If $\Xi' \ne \emptyset$, we reorder columns in $\tmop{Mat} \left( \tmop{Ada} (\Sigma), \Sigma \right)$
so that the columns corresponding to zero-nonzero conditions in 
$\Sigma_{0} \cup \Sigma_\star$ appear first. 
Then, $\tmop{Mat} \left( \tmop{Ada}
  (\Sigma), \Sigma \right)$ gets the following 
  structure
\begin{equation}\label{matrixMat_permutated} 
M=
     \left(\begin{array}{c|c}
      \begin{array}{ccc}
& & \cr & M_{1,1} & \cr & & \end {array} & M_{1,2} \\ \hline
 \begin{array}{ccc}
& & \cr & M_{2,1}  & \cr & & 
 \end {array}  & M_{2,2}
\end{array}\right) 
\end{equation}
  with
  \begin{eqnarray*}
   M_{1, 1} & = &  \tmop{Mat}( \tmop{Ada}(
    \Xi ), \Xi ),\\
    M_{1, 2} & = & \tmop{Mat}( \tmop{Ada}(
    \Xi ), \Xi' ), \\
M_{2, 1} & = & \tmop{Mat} ( ( \tmop{Ada} ( \Xi' ), i), \Sigma_{0} \cup \Sigma_\star), \\    
M_{2, 2} & = & \tmop{Mat} ( \tmop{Ada} ( \Xi' ), \Xi') .
  \end{eqnarray*}

It is easy to see that $M_{1,2}$ equals the submatrix of $M_{1,1}$ composed by the columns 
corresponding to zero-nonzero conditions in $\Sigma_{0}$ and also that all the columns
in $M_{2,1}$ corresponding to zero-nonzero conditions in $\Sigma_{0}$ are $0$. 

  Suppose, by induction hypothesis, that $\tmop{Mat} ( \tmop{Ada} (
  \Xi), \Xi)$ and $\tmop{Mat} (
  \tmop{Ada} (\Xi' ), \Xi' )$ are
  invertible. We invert $\tmop{Mat} ( \tmop{Ada} (\Sigma), \Sigma
  )$ using a Gaussian elimination method by block, i.e. multiplying to
  the left by block elementary matrices.  
We call $\tilde{\tmop{Id}}$ the submatrix of the identity matrix 
with columns indexed by the list $\Sigma_0 \cup \Sigma_\star$
composed by the columns
corresponding to zero-nonzero conditions in $\Sigma_0$. Taking into account that
$M_{2,1} \cdot \tilde{\tmop{Id}} = 0$, it is easy to check that
$$
     \left(\begin{array}{c|c}
      \begin{array}{ccc}
& & \cr & \tmop{Id} & \cr & & \end {array} & -\tilde{\tmop{Id}} \\ \hline
   \begin{array}{ccc}
& & \cr & 0 & \cr & & \end {array}  & \tmop{Id}
\end{array}\right)
     \left(\begin{array}{c|c}
      \begin{array}{ccc}
& & \cr & \tmop{Id} & \cr & & \end {array} & 0 \\ \hline
    \begin{array}{ccc}
& & \cr & 0 & \cr & & \end {array}   & M_{2,2}^{-1}
\end{array}\right)
     \left(\begin{array}{c|c}
      \begin{array}{ccc}
& & \cr & \tmop{Id} & \cr & & \end {array} & 0 \\ \hline
   \begin{array}{ccc}
& & \cr & -M_{2,1} & \cr & & \end {array} 
   & \tmop{Id}
\end{array}\right)
     \left(\begin{array}{c|c}
      \begin{array}{ccc}
& & \cr & M_{1,1}^{-1} & \cr & & \end {array} & 0 \\ \hline
  \begin{array}{ccc}
& & \cr & 0 & \cr & & \end {array}    & \tmop{Id}
\end{array}\right)
$$
is the inverse of the matrix $M$ defined in (\ref{matrixMat_permutated}). Therefore
 $\tmop{Mat} ( \tmop{Ada} (
    \Sigma), \Sigma )$ is invertible as we wanted to prove.
\end{proof}

A last key observation which leads to a well controlled size of the data we manipulate is 
the following.
Let $\mathcal{P}_i=P_1,\ldots,P_i$ and suppose that we compute
inductively 
$\tmop{Feas}( \mathcal{P}_i, Z)$ for $i = 0, \dots, s$. Since 
$1 = \tmop{card}(\tmop{Feas}( \mathcal{P}_0, Z))$ and 
$\tmop{card}(\tmop{Feas}( \mathcal{P}_s, Z)) = \tmop{card}(\tmop{Feas}( \mathcal{P}, Z)) \le \tmop{card}(Z)$, 
in the sequence
$$
\tmop{card}(\tmop{Feas}( \mathcal{P}_0, Z)) \le \dots \le
\tmop{card}(\tmop{Feas}( \mathcal{P}_i, Z)) \le \dots 
\le \tmop{card}(\tmop{Feas}( \mathcal{P}_s, Z))
$$
we 
have at most $\tmop{card}(Z)-1$ places where 
a
strict inequality holds. 
For $i = 1, \dots, s$ such that 
$$\tmop{card}(\tmop{Feas}( \mathcal{P}_{i-1}, Z)) = 
\tmop{card}(\tmop{Feas}( \mathcal{P}_{i}, Z)),$$
we have that every zero-nonzero condition in 
$\tmop{Feas}( \mathcal{P}_{i-1}, Z)$ can be extended to a zero-nonzero condition 
in 
$\tmop{Feas}( \mathcal{P}_{i}, Z)$ in only one way. 
Once this information is known, 
we have that 
for each point in $Z$,
the invertibility of $P_i$ is determined from 
the invertibility of the polynomials in $\mathcal{P}_{i-1}$ at this point. 
With this remark in mind,
we introduce the following definitions
which will be useful in
Algorithm {\rm \texttt{Zero-nonzero Determination}}
and its complexity analysis.

\begin{defn}
\tmtextbf{\tmtextup{[Compressed set of indices]}} \label{def:comp} 
  Let $I \subset \{1, \dots, s\}$ and  $\Sigma \subset \left\{ 1, 0 \right\}^I$,
  $\Sigma \neq \emptyset$.
The {\tmstrong{compressed set of indices}} $\tmop{comp} \left( \Sigma
\right)$ is defined by induction as follows:
\begin{itemize}
\item If $I = \emptyset$ define 
 $\tmop{comp}(\Sigma) =  \emptyset$.

\item If $I \ne \emptyset$, consider $i= \max ( I)$,  $I' = I \setminus \{i\}$ and $\Xi=\Sigma_{I'}$; then 
\begin{itemize}
 \item if $\tmop{card} ( \Sigma) = \tmop{card} ( \Xi
)$, define $\tmop{comp} ( \Sigma ) = \tmop{comp} (
\Xi )$,
\item if $\tmop{card} ( \Sigma) > \tmop{card} ( \Xi
)$, define $\tmop{comp}( \Sigma ) = \tmop{comp} (
\Xi ) \cup \{i\}$.
\end{itemize}
\end{itemize}
We define also the {\tmstrong{compressed list of zero-nonzero conditions}}, 
$\tmop{Comp}(\Sigma )=\Sigma_{\tmop{comp}( \Sigma )}$.
\end{defn}

\begin{example}
\label{examplecomp}
  
 If $I = \{ 1, 2, 3, 4, 5 \}$ and $\Sigma$ is the list
  of zero-nonzero conditions 
\[ 
[1\,0\,1\,1\,0, \ 1\,0\,1\,1\,1, \ 1\,1\,0\,1\,1, \ 1\,1\,1\,0\,0, \ 1\,1\,1\,0\,1]
\]
then $\tmop{comp} \left( \Sigma \right)$ is $\{2, 3, 5\}$
and $\tmop{Comp}( \Sigma )$
is 
\[ 
[0\,1\,0, \ 0\,1\,1, \ 1\,0\,1, \ 1\,1\,0,\ 1\,1\,1].
\]

\end{example}

\begin{remark}\label{rem:comp}
Let $I \subset \{ 1, \ldots, s \}$ and $\Sigma \subset \{
  0, 1\}^I$.  
Following Definition \ref{def:adap}
and  Definition \ref{def:comp}, it is easy to prove by induction in $\tmop{card}(I)$
that 
 $\tmop{Ada} \left(
    \Sigma \right)=\tmop{Ada} \left(
    \tmop{Comp} \left(\Sigma\right) \right)$.
\end{remark}

\subsection{Algorithms and complexity}

Given $I \subset \{1, \dots, s\}$ and $\Sigma \subset \{0, 1\}^I$, 
we 
represent $\Sigma$ with a $\tmop{card}(\Sigma) \times \tmop{card}(I)$ matrix filled with $0$ and $1$, 
each row representing a zero-nonzero condition in $\{0, 1\}^I$. 
We consider this representation even when $\Sigma = \emptyset$ or $I = \emptyset$. 
Similarly, we represent a list $A$ of subsets of $I$ with  a 
$\tmop{card}(A) \times \tmop{card}(I)$ matrix filled with $0$ and $1$, 
each row representing an element of $A$, and the bit $0$ (resp.$1$) in each column 
indicating that the corresponding
element does not belong (resp. belongs) to the subset of $I$. 
When we speak
of $\Sigma$ (resp. $A$)
we mean either the list $\Sigma$ (resp. $A$) or its representation as a matrix
as convenient.

For a matrix $M$ of size $m \times n$, given ordered lists of integers with no repetitions
$\ell$ and $\ell'$, with
the entries of $\ell$ (resp. $\ell'$) in $\{1, \dots, m\}$ (resp. 
$\{1, \dots, n\}$), we denote by $M(\ell, \ell')$ the submatrix
of $M$ obtained by extracting from $M$ the rows in $\ell$ and the columns in $\ell'$. 
We use this notation even when one of the lists $\ell$ and $\ell'$ is empty.
For a vector $v$ of size $m$, analoguously we denote by 
$v(\ell)$  the subvector formed by the entries with index in $\ell$.

Up to the end of the subsection, we follow the notation in 
Definition \ref{def:adap} and
Proposition \ref{prop:adaptedadapted}. 
We introduce an auxiliary definition. 

\begin{defn}\label{defn_Info} Let $I \subset \{1, \dots, s\}$ and $\Sigma \subset \{0, 1\}^I$. The matrix 
$\tmop{Info}(\Sigma)$ is defined by induction as follows:
\begin{itemize}
\item If $I = \emptyset$ define $\tmop{Info} \left( \Sigma
\right)$ as the matrix with as many rows as elements in $\Sigma$ (possibly $0$ or $1$) and $0$ columns.  

\item  If $I \ne \emptyset$, we consider 
the list
 $\ell_0$ (resp. $\ell_1$, $\ell_\star$)
formed by the indices of zero-nonzero conditions in $\Sigma$
which belong to $\Sigma_0$ (resp.  $\Sigma_1$, $\Sigma_\star$) and define  
$\tmop{Info}(\Sigma)$ as follows: 
$$
\begin{cases}
\tmop{Info}(\Sigma)(\ell_0 \cup \ell_\star,  [1, \dots, \tmop{card}(I) - 1]) = 
\tmop{Info}(\Xi), \\
\tmop{Info}(\Sigma)(\ell_1,  [1, \dots, \tmop{card}(I) - 1]) = 
\tmop{Info}(\Xi'), \\
\tmop{Info}(\Sigma)(j, \tmop{card}(I)) = 0 & \hbox{ if } j \in \ell_0, \\
\tmop{Info}(\Sigma)(j, \tmop{card}(I)) = 1 & \hbox{ if } j \in \ell_1, \\
\tmop{Info}(\Sigma)(j, \tmop{card}(I)) = \star & \hbox{ if } j \in \ell_\star. \\
\end{cases}
$$
\end{itemize}
\end{defn}

\begin{example}
 \label{exampleinfo}
Continuing Example   \ref{examplecomp}, 
$\tmop{Info}(\Sigma)$
is
$$\left(\begin{array}{ccccc}
 \star & 0 & \star & \star & 0 \cr
 \star&0&\star&\star&1 \cr 
\star&1&0&\star&\star \cr 
\star&\star&1&\star&0 \cr \star&1&\star&\star&1 
  \end{array}\right).
$$
\end{example}

The availability of the information provided by the matrix $\tmop{Info}(\Sigma)$ is very important
to obtain the complexity bound in the  algorithms in this subsection. 
We consider then the following auxiliary technical algorithm.

\medskip

\textbf{Algorithm} \texttt{Get Info}

\begin{itemize}
\item \textbf{Input:} A list $\Sigma \subset \{0,1 \}^I$ with $I \subset \{1, \dots, s\}$.

\item \textbf{Output:}  The matrix $\tmop{Info}(\Sigma)$. 
\end{itemize}

It is easy to 
give a procedure for 
Algorithm \texttt{Get Info} with bit complexity 
 $O(\tmop{card}(\Sigma)\tmop{card}(I)^2)$. 
Since the list $\Sigma$ is ordered,  this procedure takes $O(\tmop{card}(\Sigma)\tmop{card}(I))$ bit operations to compute
the lists $\ell_0, \ell_1$ and $\ell_\star$ and then does recursive calls to itself to compute 
the matrices $\tmop{Info}(\Xi)$
and $\tmop{Info}(\Xi')$.

The following algorithm computes the adapted family for a given list of zero-nonzero conditions. 

\medskip

{\noindent}\tmtextbf{Algorithm} \texttt{Adapted Family}
\begin{itemize}
\item \tmtextbf{Input:} 
A list $\Sigma \subset \{0,1 \}^I$ with $I \subset \{1, \dots, s\}$ and 
the matrix $\tmop{Info}(\Sigma)$. 
  
\item \tmtextbf{Output:} The set $\tmop{Ada} (\Sigma)$ as a list.
  
\item \tmtextbf{Procedure:} Extract from each row of  $\tmop{Info}(\Sigma)$ the subset of indices with entry $1$.
\end{itemize}

\begin{lemma}
\label{lem:adapted} Given a list $\Sigma \subset \left\{ 1, 0 \right\}^I$ with $I \subset \{1, \dots, s\}$ 
and the matrix  $\tmop{Info}(\Sigma)$,
Algorithm 
\emph{\texttt{Adapted Family}} computes the set $\tmop{Ada} (\Sigma)$ as a list. The bit complexity of this 
algorithm is 
 $O( \tmop{card} ( \Sigma )  \tmop{card}( I))$.
\end{lemma}

\begin{proof}{Proof:}
The correctness of the algorithm follows from Definition \ref{def:adap} and 
Definition \ref{defn_Info}. 
The bound on the bit complexity is clear since  $\tmop{Info}(\Sigma)$ is read once. 
\end{proof}

\begin{example}
 \label{exampleadap}
 Continuing Example  \ref{exampleinfo},
the representation of $\tmop{Ada} (\Sigma)$ 
as a list obtained by Algorithm \emph{\texttt{Adapted Family}} is 
$$\left(\begin{array}{ccccc}
 0 & 0 & 0 & 0 & 0 \cr
 0&0&0&0&1 \cr 
0&1&0&0&0 \cr 
0&0&1&0&0 \cr 0&1&0&0&1 
  \end{array}\right).
$$
Finally 
$\tmop{Ada} (\Sigma) =\{\emptyset,\{5\}, \{2\}, \{3\}, \{2,5\}\}$.

\end{example}

From now on, for every $I \subset \{1, \dots, s\}$ and $\Sigma \subset \{0, 1\}^I$, we
considered the set $\tmop{Ada}(\Sigma)$ ordered into a list as obtained by Algorithm
\texttt{Adapted Family}. 

We give now a specific method for solving the linear systems arising in the 
algorithm for 
zero-nonzero determination.

\medskip

{\noindent}\tmtextbf{Algorithm  } \texttt{Linear Solving}

\begin{itemize}
\item \tmtextbf{Input:}  
A list $
\Sigma \subset \{0, 1\}^I$ with $I \subset \{1, \dots, s\}$, 
the matrices $\tmop{Info}(\Sigma)$ and $\tmop{Mat}(\tmop{Ada}(\Sigma), \Sigma)$ 
and an integer 
vector $v$ of size
$\tmop{card}(\Sigma)$. 
  
\item \tmtextbf{Output:} The 
vector $c = \tmop{Mat}(\tmop{Ada}(\Sigma), \Sigma)^{-1}v$. 
 
\item \tmtextbf{Procedure:}

\begin{enumerate}

\item If $I = \emptyset $ output $c = v$. 
So from now we suppose $I \ne \emptyset$. 

\item \label{ls:step_2} Extract from $\tmop{Info}(\Sigma)$ 
the lists 
$\ell_0, \ell_1$ and $\ell_\star$ (cf. Definition \ref{defn_Info}),
define $\ell=\ell_0 \cup \ell_\star$, 
$m^0 = \tmop{card}(\ell_0)$ and  $m^\star = \tmop{card}(\ell_\star)$.

\item \label{ls:step_3} Compute  $t(\ell)
= 
\tmop{Mat}(\tmop{Ada} (\Xi), \Xi)^{-1}
v(\ell)$,  
doing a recursive call to Algorithm \texttt{Linear Solving}. 
  
\item \label{ls:step_4} If $m^0 \ne 0$ and $m^\star \ne 0$, 
compute
$t(\ell_1) 
= 
-\tmop{Mat}( (\tmop{Ada}(\Xi'), i), \Sigma_\star)t(\ell_\star) + 
v(\ell_1)$.

\item \label{ls:step_5} If $m^0 \ne 0$: 
\begin{enumerate}
\item \label{ls:step_5:a} Compute $c(\ell_1)
= 
\tmop{Mat}(\tmop{Ada} (\Xi'), \Xi')^{-1}
t(\ell_1)$,  
doing a recursive call to Algorithm \texttt{Linear Solving}. 
  
\item \label{ls:step_5:b} Compute $c
(\ell_0) =  
t(\ell_0)-t(\ell_1)$.

\item \label{ls:step_5:c} Define $c(\ell_\star) = t(\ell_\star)$. 

\end{enumerate}

\item \label{ls:step_6} Output $c$. 

\end{enumerate}

\end{itemize}

\begin{remark} Let  $I \subset \left\{ 1, \ldots, s \right\}$
  and $\Sigma \subset \left\{ 1, 0 \right\}^I$. Following the steps of the algorithm, it is easy to prove by induction in 
$\tmop{card}(I)$ that if $v$ is an integer vector of size $\tmop{card}(\Sigma)$, then 
 $c=\tmop{Mat}(\tmop{Ada} ( \Sigma), \Sigma)^{-1}v$ is also an integer vector. 
\end{remark}

\begin{proposition}\label{pro:adaptedadaptedquick}
Given a  list $\Sigma \subset \{0, 1\}^I$ with $I \subset \{1, \dots, s\}$, 
the matrices $\tmop{Info}(\Sigma)$ and $\tmop{Mat}(\tmop{Ada}(\Sigma), \Sigma)$ 
and an integer vector $v$ of size
$\tmop{card}(\Sigma)$, Algorithm \emph{\texttt{Linear Solving}} computes the 
vector
$$c = \tmop{Mat}(\tmop{Ada}(\Sigma), \Sigma)^{-1}v.$$

If $Z \subset \mathup{C}^k$  is a finite set with $r$ elements, 
$\mathcal{Q}$ is a finite set of polynomials indexed by $I$ and
$\tmop{Feas}(\mathcal{Q} , Z ) \subset \Sigma$ then
$v = \tmop{Qu}( \mathcal{Q}^{\tmop{Ada} ( \Sigma) }, Z )$ implies 
$c = c(\Sigma, Z)$ and the bit complexity of Algorithm \emph{\texttt{Linear Solving}}  is 
$O(\tmop{card}(\Sigma)\tmop{card}(I) + \tmop{card} (\Sigma)^2 \tmop{bit}(r))$. 
\end{proposition}

\begin{proof}{Proof:}
The correctness of the algorithm follows from the 
formula for the inverse of $\tmop{Mat}(\tmop{Ada}(\Sigma), \Sigma)$
coming from the proof of Proposition 
\ref{prop:adaptedadapted}. 
The fact that $v = \tmop{Qu}( \mathcal{Q}^{\tmop{Ada} ( \Sigma) }, Z )$ implies 
$c = c(\Sigma, Z)$ follows from Proposition \ref{pro:matrix of signs}. 
Now we deal with the complexity analysis. 
Note that the matrices $\tmop{Info}(\Xi), \tmop{Mat}(\tmop{Ada}(\Xi), \Xi),$
$\tmop{Info}(\Xi')$ and $\tmop{Mat}(\tmop{Ada}(\Xi'), \Xi')$, which 
are necesary to do the recursive calls, and the matrix 
$\tmop{Mat}((\tmop{Ada}(\Xi'), i), \Sigma_\star)$, which is necesary at Step \ref{ls:step_4}
can be extracted from the matrices $\tmop{Info}(\Sigma)$ and 
$\tmop{Mat}(\tmop{Ada}(\Sigma), \Sigma)$.

A rough description of Step \ref{ls:step_2} is the following:  
consider first all the lists we want to determine as empty and then, 
reading the last column of $\tmop{Info}(\Sigma)$ one row at a time, 
actualize the right lists and their lengths. 
Taking into account that increasing $n$ times the quantity $0$  by $1$  takes
$O(n)$ bit operations, this step takes at most $C_1 \tmop{card}(\Sigma)$ bit operations
for some constant $C_1$. 

At Step \ref{ls:step_3}, the recursive call to Algorithm \texttt{Linear Solving}
is done for the list $\Xi$ and the vector 
$\tmop{Qu}({\cal Q}^{\tmop{Ada}(\Xi)}, Z)$ and since
$\tmop{Feas} (\mathcal{Q} , Z ) \subset \Sigma$, we have that 
$\tmop{Feas} (\mathcal{Q'} , Z ) \subset \Xi$ where $\mathcal{Q'}$ is the set
obtained from $\mathcal{Q}$ by removing its element of index $i$.

Step \ref{ls:step_4} takes first $2$ bit operations to decide if $m^0 \ne 0$ and $m^\star \ne 0$. 
Since
$v(\ell_1) = \tmop{Qu}({\cal Q}^{(\tmop{Ada}(\Xi'), i)}, Z)$,  
all its entries are nonnegative integers less than or equal to 
$r$. The product $\tmop{Mat}( (\tmop{Ada}(\Xi'), i), \Sigma_\star)t(\ell_\star)$ 
is computed
by reading if the elements in $\tmop{Mat}( (\tmop{Ada}(\Xi'), i), \Sigma_\star)$ are $0$ or $1$
and adding the corresponding elements from $t(\ell_\star)$. 
Since 
$t(\ell_\star)$
is a subvector of $c(\Xi, Z)$, 
all its entries are nonnegative integers and their sum is less than or equal to $r$. 
So, we conclude that this step takes at most $2 + C_2 m^0 m^\star \tmop{bit}(r)$ bit operations for
some constant $C_2$.

The begining of Step \ref{ls:step_5} takes $1$ bit operation to decide if $m^0 \ne 0$. 
It is easy to see that at Step \ref{ls:step_5}(a), the recursive call to 
Algorithm \texttt{Linear Solving}
is done for the list $\Xi'$ and the vector 
$\tmop{Qu}({\cal Q}^{\tmop{Ada}(\Xi')}, Z^1)$ where 
$$Z^1 
= \bigcup_{\sigma \in \Sigma_1} \tmop{Reali}(\sigma, Z)
.$$ 
Also, by definition of $Z^1$, we have that 
$\tmop{Feas} (\mathcal{Q}' , Z^1 ) \subset \Xi'$.
On the other hand, it is also easy to see that Step \ref{ls:step_5}(b)
takes $C_3m^0\tmop{bit}(r)$ bit operations for some constant $C_3$.

After this analysis, the bound on the bit complexity can be proved by induction 
in $\tmop{card}(I)$. 
\end{proof}

\begin{notation}
Let $J\subset \{1,\ldots,s\}$ and $j \in \{1, \dots, s\}$ bigger than $\max(J)$. 
For a zero-nonzero condition $\sigma \in \{0, 1\}^J$, we denote by 
$\sigma \wedge 0 \in \{0, 1\}^{J \cup \{j\}}$ (resp. 
$\sigma \wedge 1$) the zero-nonzero condition 
obtained by extending $\sigma$ with $\sigma(j) = 0$ (resp. $\sigma(j) = 1$). 
For a list 
$\Sigma = [\sigma_1, \dots, \sigma_n] \subset \{0, 1\}^J$ of zero-nonzero conditions, we
denote by 
$\Sigma \wedge \{0, 1\}^{\{j\}} \subset \{0, 1\}^{J \cup \{j\}}$ the 
list of zero-nonzero conditions 
$[\sigma_1 \wedge 0, \sigma_1 \wedge 1, \dots, \sigma_n \wedge 0, \sigma_n \wedge 1]
$ and by
$\Sigma \wedge \{0\}^{\{j\}} \subset \{0, 1\}^{J \cup \{j\}}$
(resp. $\Sigma \wedge \{1\}^{\{j\}}$)
 the 
list of zero-nonzero conditions 
$
[\sigma_1 \wedge 0,  \dots, \sigma_n \wedge 0]
$ (resp. $[\sigma_1 \wedge 1,  \dots, \sigma_n \wedge 1]$).

\end{notation}

We are now ready for our main algorithm. 
This algorithm 
determines
iteratively, for $i = 0, \dots, s$ the list 
$\Sigma_i := \tmop{Feas} ( \mathcal{P}_i, Z)$  of
the  zero-nonzero conditions realized by $\mathcal{P}_i$ on $Z$, and the
corresponding list $c_i := c ( \mathcal{P}_i, Z)$ of cardinals. In order to do so within a good 
complexity bound, the algorithm also computes at each step 
the compressed set of indices $\tmop{comp}_i := \tmop{comp}(\Sigma_i)$, 
the matrix $\tmop{Info}_i := \tmop{Info}(\tmop{Comp}(\Sigma_i))$, 
the list $\tmop{Ada}_i := \tmop{Ada}(\Sigma_i)$
and 
the matrix $\tmop{Mat}_i := \tmop{Mat}(\Sigma_i,  \tmop{Ada}(\Sigma_i))$.

\medskip

{\noindent}\tmtextbf{Algorithm  } \texttt{Zero-nonzero
Determination}

\begin{itemize}
  \item \tmtextbf{Input:} a finite subset $Z \subset \mathup{C}^k$ with $r$
  elements and a finite list $\mathcal{P} = P_1, \ldots, P_s$ of polynomials
  in $\mathup{L} [X_1, \ldots, X_k]$.
  
  \item \tmtextbf{Output:} the list $\tmop{Feas} ( \mathcal{P}, Z)$  of
the  zero-nonzero conditions realized by $\mathcal{P}$ on $Z$, and the
  corresponding list $c ( \mathcal{P}, Z)$ of cardinals.
  
  \item \tmtextbf{Blackbox:} for a polynomial $Q \in \mathup{L} [X_1, \ldots, X_k]$, the 
invertibility-query
  blackbox $\tmop{Qu} (Q, Z) .$
  
 \item \tmtextbf{Procedure:}
    \begin{enumerate}

\item Compute ${r = \tmop{Qu} (1, Z)}$ using the invertibility-query
  blackbox. If $r = 0$, output $\tmop{Feas} ( \mathcal{P}, Z) = \emptyset$ and  
$c( \mathcal{P}, Z) = \emptyset$. So from now  we suppose $r > 0$.

\item Initialize $\Sigma_0 = [\emptyset]$, $c_0 = [r]$, 
$\tmop{comp}_0 = \emptyset$, 
$\tmop{Info}_0$ as the matrix with $1$ 
row and $0$ columns,
$\tmop{Ada}_0 = [\emptyset]$ and $\tmop{Mat}_0 = (\, 1 \, )$.

\item \label{step:3:ma}   For $i$ from 1 to $s$:
  \begin{enumerate}
    
\item Compute $\tmop{Qu} (P_i, Z)$ using the invertibility-query blackbox.
    
\item Using the equality
    \[ \label{10:eqn:BKR} \left(\begin{array}{cc}
         1 & 1\\
         0 & 1
       \end{array}\right) \cdot 
       \left(\begin{array}{c}
         c (P_i = 0, Z_{})\\
         c (P_i \ne 0, Z_{})
       \end{array}\right) = \left(\begin{array}{c}
         \tmop{Qu} (1, Z)\\
         \tmop{Qu} (P_i, Z)
       \end{array}\right) \]
compute $c (P_i = 0, Z)$ and $c (P_i \ne 0, Z)$.

\item If $c(P_i = 0, Z) = 0$ (resp. $c(P_i \ne 0, Z) = 0$), 
$\Sigma_i = \Sigma_{i-1} \wedge \{1\}^{\{i\}}$ (resp. $\Sigma_{i-1} \wedge \{0\}^{\{i\}}$), 
$c_i = c_{i-1}$, $\tmop{comp}_i = \tmop{comp}_{i-1}$,
$\tmop{Info}_i = \tmop{Info}_{i-1}$, 
$\tmop{Ada}_i = \tmop{Ada}_{i-1}$ and $\tmop{Mat}_i = \tmop{Mat}_{i-1}$.

Else if $c(P_i = 0, Z) > 0$ and $c(P_i \ne 0, Z) > 0$:

\begin{enumerate}

\item Compute $v' = \tmop{Qu} ( \mathcal{P}_i^{( \tmop{Ada}_{i-1},i)}, Z)$  using 
the    invertibility-query blackbox. 
    
\item Take the auxiliary list $\Sigma = \tmop{Comp}(\Sigma_{i-1})\wedge\{0, 1\}^{\{i\}}$
and  
determine  $\tmop{Info}(\Sigma)$ and $\tmop{Mat}(\tmop{Ada}( \Sigma), \Sigma)$. 
Consider the integer vector $v$ of size $\tmop{card}(\Sigma)$ having in its
odd entries the entries of $\tmop{Qu} ( \mathcal{P}_i^{\tmop{Ada}_{i-1}}, Z)$ 
(which has already been computed at 
previous steps) and in its even entries the entries of $v'$. 
Compute
$c = \tmop{Mat}(\tmop{Ada}( \Sigma), \Sigma)^{-1} v$,
using Algorithm \texttt{Linear Solving}.

\item Compute $c_i$ removing from $c$ its zero components. 
Compute also $\Sigma_i$ 
going trough $c$ by pairs of elements (note that each pair will have at least
one element different from zero). If both elements are different from zero, the 
corresponding zero-nonzero condition in $\Sigma_{i-1}$ is extended both with a $0$ and a $1$
in $\Sigma_i$.  If only the first (resp. second) element of the pair 
is different from zero, the 
corresponding zero-nonzero condition in $\Sigma_{i-1}$ is extended only with a $0$ (resp. $1$) 
in $\Sigma_i$. At the same time, compute the lists 
$\ell_0, \ell_1$ and $\ell_\star$.

\item  If $\tmop{card} ( \Sigma_i ) = \tmop{card} ( \Sigma_{i
      - 1})$, 
$\tmop{comp}_i = \tmop{comp}_{i-1}$, 
$\tmop{Info}_i = \tmop{Info}_{i-1}$,
$\tmop{Ada}_i = \tmop{Ada}_{i-1}$ and $\tmop{Mat}_i = \tmop{Mat}_{i-1}$.

Else if $\tmop{card} ( \Sigma_i ) > \tmop{card} ( \Sigma_{i
      - 1})$:
\begin{itemize}
\item $\tmop{comp}_i = \tmop{comp}_{i-1} \cup \{i\}$.

\item Define $\Xi'_i= \Sigma_i(\ell_1,  \tmop{comp}_{i-1}))$.

\item Compute $\tmop{Info}(\Xi'_i)$, 
using Algorithm \texttt{Get Info}, 
and extract $\tmop{Ada}(\Xi'_i)$, 
using Algorithm \texttt{Adapted Family}.

\item Compute $\tmop{Info}_i$, 
from $\tmop{Info}_{i-1}$ and  $\tmop{Info}(\Xi'_i)$
and  $\tmop{Ada}_i$ from $\tmop{Ada}_{i-1}$ and $\tmop{Ada}(\Xi'_i)$.

\item Compute
$\tmop{Mat}((\tmop{Ada}(\Xi'_i), i), \Sigma')$ 
with  $\Sigma' = \Sigma_i(\ell,  \tmop{comp}_{i}))$ and 
finally $\tmop{Mat}_i$.

\end{itemize}

\end{enumerate}
\end{enumerate}
  \item Output $\tmop{Feas} ( \mathcal{P}, Z) = \Sigma_s$ and
  $\tmop{c}( \mathcal{P}, Z) = c_s$.
  \end{enumerate}
\end{itemize}

\begin{theorem}\label{main_theorem} Given a finite 
subset $Z \subset \mathup{C}^k$ with $r$
elements and a finite list $\mathcal{P} = P_1, \ldots, P_s$ of polynomials
in $\mathup{L} [X_1, \ldots, X_k]$, Algorithm {\rm \texttt{Zero-nonzero Determination}}
computes the list $\tmop{Feas} ( \mathcal{P}, Z)$  of
the  zero-nonzero conditions realized by $\mathcal{P}$ on $Z$ and the
corresponding list $c( \mathcal{P}, Z)$ of cardinals. The complexity
of this algorithm is  
$O(s r^2 \tmop{bit}(r))$ bit operations
plus $1 + sr$ calls to the to the
invertibility-query blackbox which
are done for products of
at most $\tmop{bit}(r)$ products of polynomials in $\mathcal{P}$.
\end{theorem}

\begin{proof}{Proof:}  The correctness of the algorithm follows from 
Proposition \ref{pro:matrix of signs}, Proposition \ref{prop:adaptedadapted}   and Remark \ref{rem:comp}. 

We prove first the bound on the number of bit operations. 

Step \ref{step:3:ma}(b) and evaluating the conditions $c(P_i = 0, Z) = 0$ and $c(P_i \ne 0, Z) = 0$
at the beginning of Step \ref{step:3:ma}(c) take $O(\tmop{bit}(r))$ bit operations. 

At Step \ref{step:3:ma}(c)ii,
$\Sigma$ is obtained by duplicating each row in 
$\Sigma_{i-1}([1, \dots, \tmop{card}(\Sigma_{i-1})], \tmop{comp}_{i-1})$ and adding a final new column with 
a $0$ in the odd rows and a $1$ in the even rows, 
and $\tmop{Info}(\Sigma)$ is obtained in a 
similar way. 
The matrix  $\tmop{Mat}(\tmop{Ada}( \Sigma), \Sigma)$
is obtained from
$\tmop{Mat}_{i-1}$ following the formula 
(\ref{matrixMat_permutated}) in the proof of Proposition \ref{prop:adaptedadapted}. 
Since $\tmop{card}(\tmop{Comp}(\Sigma_{i-1})) \le r$ and $\tmop{card}(\tmop{comp}_{i-1}) \le r - 1$, 
by Proposition 
\ref{pro:adaptedadaptedquick}, the computation of $c$ takes 
$O(r^2\tmop{bit}(r))$ bit operations. 

Step \ref{step:3:ma}(c)iii takes $O(r)$ bit operations. 

Since $\tmop{card}(\tmop{comp}_{i-1}) \le \min\{s, r-1\}$ for every $i$, 
at Step \ref{step:3:ma}(c)iv 
the computation of 
$\tmop{Info}(\Xi'_i)$ takes $O(\tmop{card}(\Xi'_i)sr)$ bit operations 
using Algorithm \texttt{Get Info}. By Lemma \ref{lem:adapted},  the computation of $\tmop{Ada}(\Xi'_i)$ takes $O(\tmop{card}(\Xi'_i)r)$ 
bit operations.
The computation of 
$\tmop{Info}_i$ and $\tmop{Ada}_i$ 
is then done following Definition \ref{defn_Info} and using the already 
computed matrices 
$\tmop{Info}_{i-1}$, $\tmop{Info}(\Xi'_i)$,
$\tmop{Ada}_{i-1}$ and $\tmop{Ada}(\Xi'_i)$
 and the lists $\ell_0, \ell_1$ and $\ell_\star$. 
By evaluating the product $\sigma^J$ for every
$\sigma \in \Sigma'$ and $J \in 
(\tmop{Ada}(\Xi'_i), i)$, the computation of 
$\tmop{Mat}((\tmop{Ada}(\Xi'_i), i), \Sigma')$ takes $O(\tmop{card}(\Xi'_i)sr)$ bit operations. 
The matrix $\tmop{Mat}_i$ is obtained following the formula 
(\ref{matrixMat_permutated}) in the proof of Proposition \ref{prop:adaptedadapted}
using the already computed matrices 
$\tmop{Mat}_{i-1}$ and $\tmop{Mat}((\tmop{Ada}(\Xi'_i), i), \Sigma')$.

It is easy to prove that
$\sum_{i = 1, \dots, s} \tmop{card}(\Xi'_i) 
= \tmop{card}(\tmop{Feas}({\cal P}, Z)) - 1
\le r-1$. From this, we conclude the number of bit operations of this algorithm is $O(sr^2\tmop{bit}(r))$.

Finally we prove the assertion on the invertibility-queries to compute. 
At Step 3, for $i = 1, \dots, s$, there are $\tmop{card}( \Sigma_{i - 1} ) \leqslant r$ new
invertibility-queries to determine. Therefore, the total number of calls to the
invertibility-query blackbox is bounded by $1 + sr$. 
Since by Proposition \ref{pro:bkr} the
  elements of  $\tmop{Ada}_{i-1}$ are subsets of
  $\left\{ 1, \ldots s \right\}$ with at most $\tmop{bit} \left( r \right) -
  1$ elements, 
these  calls  
are done for polynomials which are product of
at most  $\tmop{bit}(r)$ products of polynomials in $\mathcal{P}$. 
\end{proof}

\begin{defn}
  \label{10:rem:adapted} We denote by $\tmop{Used} ( \mathcal{P}, Z)$ the list of subsets of 
$\{1, \dots, s\}$  constructed inductively as follows:
\begin{itemize}
 \item $\tmop{Used} ( \mathcal{P}_0, Z) = \emptyset$. 
\item For $1 \le i \le s$, 
$$
\begin{cases}
     \tmop{Used} ( \mathcal{P}_i, Z) = \tmop{Used} ( \mathcal{P}_{i - 1},
    Z) \cup \left\{ i \right\} & \tmop{if} c (P_i = 0, Z) = 0 \tmop{or} c (P_i \ne 0, Z) = 0,\\
     \tmop{Used} ( \mathcal{P}_i, Z) = \tmop{Used} ( \mathcal{P}_{i - 1},
    Z) \cup  ( \tmop{Ada} (\tmop{Feas}(\mathcal{P}_{i-1}, Z)), i ) & \tmop{otherwise}.
   \end{cases}
$$

\end{itemize}
\end{defn}

\begin{remark}\label{remark_used_znz}
 It is easy to see that if $Z \ne \emptyset$, 
  $\tmop{Used} ( \mathcal{P}, Z)$ is exactly the list of subsets $J$ of
  $\left\{ 1, \ldots s \right\}$ such that the invertibility-query of $P^J$
  has been computed during the execution of Algorithm \emph{\texttt{Zero-nonzero
  Determination}}. It is also clear that $\tmop{Used} ( \mathcal{P}, Z)$ can
  be determined from $\tmop{Feas} ( \mathcal{P}, Z)$. 
As mentioned in Theorem \ref{main_theorem},  
the elements
 of $\tmop{Used} ( \mathcal{P}_{}, Z)$ are subsets of $\left\{ 1, \ldots s
  \right\}$ with at most $\tmop{bit}(r)$ elements. 
\end{remark}

\section{The real-nonreal sign determination problem}
 \label{sec:nonrealqueries}
 
We recall that $\mathup{K}$ is an ordered field, $\mathup{R}$ a real closed extension of 
$\mathup{K}$ and $\mathup{C} = \mathup{R}[i]$, $Z \subset \mathup{C}^k$ a finite set and 
${\cal P} = P_1, \dots, P_s$ a list of polynomials in $\mathup{K}[X_1, \dots, X_k]$.

Since $\tmop{Feas} ( \mathcal{P},
Z_{\mathup{C} \setminus \mathup{R}}) \subset 
\tmop{Feas} ( \mathcal{P},Z)$ 
and for every $\sigma \in \tmop{Feas} ( \mathcal{P},Z)$ we have that
\[
c( \sigma, Z_{\mathup{C} \setminus \mathup{R}}) = c(\sigma, Z ) - c(\sigma, Z_{\mathup{R}} ) = 
c(\sigma, Z ) - 
\sum_{\tau \in \Gamma_{\sigma}} 
c_{\rm sign}( \tau,Z_{\mathup{R}}) 
\]
where
 \[ \Gamma_{\sigma} = \left\{ \tau \in \tmop{Feas}_{\rm sign} ( \mathcal{P},
     Z_{\mathup{R}}) \mid \{i \mid \tau \left( i \right) = 0\} =
\{i \mid \sigma \left( i \right) = 0\}\right\}, \]
it is easy to combine algorithms for sign and zero-nonzero determination
to solve the real-nonreal sign determination problem as follows. 

\medskip

{\noindent}\tmtextbf{Algorithm  } \texttt{Real-nonreal Sign Determination}

\begin{itemize}
  \item \tmtextbf{Input:} a finite subset $Z \subset \mathup{C}^k$ with $r$
  elements and a finite list $\mathcal{P} = P_1, \ldots, P_s$ of polynomials
  in $\mathup{K} [X_1, \ldots, X_k]$.
  
  \item \tmtextbf{Output:} the lists
$\tmop{Feas}_{\rm sign} ( \mathcal{P}, Z_{\mathup{R}})$ and
 $\tmop{Feas} ( \mathcal{P}, Z_{\mathup{C}\setminus\mathup{R}})$  of
sign and zero-nonzero conditions realized by $\mathcal{P}$ on 
$Z_{\mathup{R}}$ and $Z_{\mathup{C}\setminus\mathup{R}}$ respectivelly, and the
  corresponding lists $c_{\rm sign}( \mathcal{P}, Z_{\mathup{R}})$
and $c(\mathcal{P}, Z_{\mathup{C}\setminus\mathup{R}})$ of cardinals.
  
\item \tmtextbf{Blackbox 1:} for a polynomial $Q \in \mathup{K}[X_1, \ldots, X_k]$, the 
Tarski-query
  blackbox $\tmop{Qu} (Q, Z_{\mathup{R}}) .$

\item \tmtextbf{Blackbox 2:} for a polynomial $Q \in \mathup{K}[X_1, \ldots, X_k]$, the 
invertibility-query
  blackbox $\tmop{Qu} (Q, Z) .$

 \item \tmtextbf{Procedure:}
    \begin{enumerate}

\item Compute 
$\tmop{Feas}_{\rm sign} ( \mathcal{P}, Z_{\mathup{R}})$ and
$c_{\rm sign}( \mathcal{P}, Z_{\mathup{R}})$, as explained in \cite[Chapter 10]{BPRcurrent}.

\item Compute $\tmop{Feas}(\mathcal{P}, Z)$
and $c(\mathcal{P}, Z)$, using Algorithm \texttt{Zero-nonzero Determination}.

\item For $\sigma \in \tmop{Feas}( \mathcal{P},Z)$ 
compute 
$
c( \sigma, Z_{\mathup{C} \setminus \mathup{R}}) = 
c(\sigma, Z ) - 
\sum_{\tau \in \Gamma_{\sigma}} 
c_{\rm sign}( \tau,Z_{\mathup{R}}) 
$.

\item Define $\tmop{Feas}(\mathcal{P}, Z_{\mathup{C}\setminus\mathup{R}})$ as the list
of $\sigma \in \tmop{Feas}( \mathcal{P},Z)$ 
such that 
$c( \sigma, Z_{\mathup{C} \setminus \mathup{R}}) \ne 0$ and 
$c(\mathcal{P}, Z_{\mathup{C}\setminus\mathup{R}})$ as the corresponding list of cardinals.

\end{enumerate}
\end{itemize}

From the results in Section \ref{sec:comp} and \cite[Chapter 10]{BPRcurrent}, we deduce the following result. 

\begin{theorem}\label{2nd_theorem} Given a finite 
subset $Z \subset \mathup{C}^k$ with $r$
elements and a finite list $\mathcal{P} = P_1, \ldots, P_s$ of polynomials
in $\mathup{K} [X_1, \ldots, X_k]$, Algorithm \emph{\texttt{Real-nonreal Sign Determination}}
computes 
the lists
$\tmop{Feas}_{\rm sign} ( \mathcal{P}, Z_{\mathup{R}})$ and
 $\tmop{Feas} ( \mathcal{P}, Z_{\mathup{C}\setminus\mathup{R}})$  of
sign and zero-nonzero conditions realized by $\mathcal{P}$ on 
$Z_{\mathup{R}}$ and $Z_{\mathup{C}\setminus\mathup{R}}$ respectivelly, and the
  corresponding lists $c_{\rm sign}( \mathcal{P}, Z_{\mathup{R}})$
and $c(\mathcal{P}, Z_{\mathup{C}\setminus\mathup{R}})$ of cardinals.
The complexity
of this algorithm is  
$O(sr^2 \tmop{bit}(r))$ bit operations
plus $1 + 2sr$ calls to the to the
Tarski-query blackbox which
are done for products of
at most $\tmop{bit}(r)$ products of polynomials in $\mathcal{P}$ or squares of polynomials in 
$\mathcal{P}$,
plus $1 + sr$ calls to the to the
invertibility-query blackbox which
are done for products of
at most $\tmop{bit}(r)$ products of polynomials in $\mathcal{P}$.
\end{theorem}

\section{How to compute the queries?}\label{sec:queries}

For completeness, in this section we summarize the main methods to compute the invertibility-queries and 
Tarski-queries; all these methods are in fact closely related between them. We refer to \cite{BPR}
for notations, proofs and details. 

We deal first with the univariate case. Suppose that $Z \subset \mathup{C}$ is given 
as the zero set of a polynomial $P \in \mathup{K}[X]$. For simplicity, we assume $P$ to be monic. 
The invertibility-queries and Tarski-queries can be computed as follows. 
\begin{itemize}
  \item through the Sturm sequence $\tmop{Stu}(P, Q)$,  which is a slight
  modification of the sequence of remainders of $P$ and $P'Q$: 
  \begin{eqnarray*}
    \tmop{Qu} (Q, Z) & = & \deg( P ) - \deg( \gcd ( P,
    P' Q ) ),\\
    \tmop{TaQu} (Q, Z_{\mathup{R}}) & = & \tmop{var} ( 
\tmop{Stu}(P, Q); - \infty, \infty),
  \end{eqnarray*}
where $\tmop{var}$ is the number of sign variations in the corresponding sequence.
Note that $\gcd(P, P' Q )$ is the last element in $\tmop{Stu}(P, Q)$.
   
\item through the subresultant sequence $\tmop{sRes}(P, R)$ of $P$ and 
$R = \tmop{Rem}( P, P' Q)$:
  \begin{eqnarray*}
    \tmop{Qu} (Q, Z) & = & \deg( P ) - \deg( \gcd ( P,
    P' Q ) ),\\
    \tmop{TaQu} (Q, Z_{\mathup{R}}) & = & \tmop{PmV}( \tmop{sRes}
    ( P, R )),
  \end{eqnarray*}
  where 
  $\tmop{PmV}$ is a generalization of the difference between the number of
  positive elements and negative elements in the corresponding sequence.
Note that $\gcd(P, P' Q )$ is the last nonzero element in $\tmop{sRes}(P, R)$.

    \item through the Bezoutian matrix $\tmop{Bez}( P, R)$ of $P$ and  
$R = \tmop{Rem}( P, P' Q)$:
  \begin{eqnarray*}
    \tmop{Qu} (Q, Z) & = & \tmop{rank}( \tmop{Bez}( P, R)),\\
    \tmop{TaQu} (Q, Z_{\mathup{R}}) & = & \tmop{sign}( \tmop{Bez}( P, R)).
  \end{eqnarray*}
\end{itemize}

Now we deal with the multivariate case (including the univariate case). 
Suppose that $Z \subset \mathup{C}^{k}$ is given as the zero set of a polynomial system in 
$\mathup{K}[X_1,\dots, X_k]$. The invertibility-queries and Tarski-queries can be computed as follows.
\begin{itemize}
  \item through the Hermite matrix $\tmop{Her}(P, Q)$:
  \begin{eqnarray*}
    \tmop{Qu} (Q, Z) & = & \tmop{rank} (\tmop{Her}( P, Q)),\\
   \tmop{TaQu} (Q, Z_{\mathup{R}}) & = & \tmop{sign} (\tmop{Her}
    ( P, Q )).
  \end{eqnarray*}
\end{itemize}

For the univariate case, we deduce the following result:

\begin{theorem}

\begin{enumerate} 
\item \label{first_item} Let $P \in \mathup{L}[X]$ 
and ${\cal P} = P_1, \dots, P_s$ be a finite list of polynomials in 
$\mathup{L}[X]$ such that the degree of $P$ and the polynomials in ${\cal P}$ 
is bounded by $d$.
The complexity of zero-nonzero determination is 
$\tilde{O}(sd^2)$
bit operations plus 
$\tilde{O}(sd^2)$
arithmetic operations in $\mathup{L}$.
If $P$ and the polynomials in ${\cal P}$ lie in $\mathbb{Z}[X]$ and the bit size of the 
coefficients of all these polynomails is bounded by $\tau$, the complexity of
zero-nonzero determination is 
$\tilde{O}(\tau s d^3)$ 
bit operations.

\item Let $P \in \mathup{K}[X]$ 
and ${\cal P} = P_1, \dots, P_s$ be a finite list of polynomials in 
$\mathup{K}[X]$ such that the degree of $P$ and the polynomials in ${\cal P}$ 
is bounded by $d$.
The complexity of real-nonreal sign determination is 
$\tilde{O}(sd^2)$
bit operations plus 
$\tilde{O}(sd^2)$
arithmetic operations in $\mathup{K}$.
If $P$ and the polynomials in ${\cal P}$ lie in $\mathbb{Z}[X]$ and the bit size of the 
coefficients of all these polynomails is bounded by $\tau$, the complexity of
real-nonreal sign determination is 
$\tilde{O}(\tau s d^3)$
bit operations.
\end{enumerate}
\end{theorem}

\begin{proof}{Proof:}  To prove item \ref{first_item}, we estimate the complexity of computing
the invertibility-queries. 
At each call we have to compute the product between an already computed product of 
$P'$ and at most $\tmop{bit}(d)-1$ polynomials in ${\cal P}$ and a polynomial in ${\cal P}$; and then 
we have to compute the gcd 
between this polynomial and $P$.
The complexity of these computations
is 
$\tilde O(d)$ operations in $\mathup{L}$;  
in the case of integer coefficients, 
the complexity of these 
computations is $\tilde O(\tau d^2)$  bit operations (see 
\cite[Chapters 8 and 11]{vzG}). 
The result then follows directly from Theorem \ref{main_theorem}.

Item 2 is proved in a similar way, the Tarski-queries being computed 
following the subresultant 
sequence aproach as in \cite{LiR}.  
The Tarski-query of a product of
at most $2\tmop{bit}(d)$ polynomials in ${\cal P}$ for the real zero set of $P$ 
can be computed within $\tilde O(d)$ operations in $\mathup{K}$; 
in the case of integer coefficients, 
the complexity of this 
computation is $O(\tau d^2)$ bit operations (see \cite{Rei} and \cite{LiR}). 
The result then follows directly from Theorem \ref{2nd_theorem}.
\end{proof}

\section{Real-nonreal sign determination with parameters}\label{sec:parameters}

Let $P \in \mathup{K}[ Y_1, \dots, Y_m, X]$ with $P$ monic with
respect to $X$ and $\mathcal{P} = P_1, \dots, P_s \subset \mathup{K}[Y_1,
\dots, Y_m, X]$, where $Y_1,
\ldots, Y_m$ are parameters and $X$ is the main variable. 
Our goal is to describe a family of polynomials in $\mathup{K}[Y_1, \dots, Y_m]$ such that
the result of the 
real-nonreal sign determination problem  after 
specialization of the parameters  
is determined by a sign condition on this family. 
First, we introduce some notation. 

\begin{notation}
We denote by
$\tmop{Prod}_{\tmop{bit} \left( p \right)} (
\mathcal{P})$ the family of polynomials which are of the
form
\[ {\mathcal{P}}^{\alpha} = \prod_{1 \le i \le s}{P_i^{\alpha_i}} \]
with $\alpha = (\alpha_1, \dots, \alpha_s) \in \{0, 1, 2\}^{\{1, \ldots, s\}}$ such 
that $\tmop{card}\{i \ | \ \alpha_i \neq 0\} \le \text{\tmtextrm{\tmop{bit}}} (p)$.
For $Q\in \mathup{K}[ Y_1, \dots, Y_m, X]$, we denote by $\tmop{HMin} (P, Q)$
the subset of $\mathup{K}[ Y_1, \dots, Y_m]$ formed by the principal minors of 
$\tmop{Her}(P,Q)$. 
Finally, we denote by
\[ \tmop{HElim}( P, \mathcal{P}) = \bigcup_{Q \in
   \tmop{Prod}_{\tmop{bit} ( p)} (
   \mathcal{P})}\tmop{HMin} (P, Q).
\]
\end{notation}

Now we can state our result. 

\begin{theorem}
For $y \in \mathup{R}^m$, let $Z(y) \subset \mathup{C}$ be the zero set 
of $P(y_1, \dots, y_m, X)$ and
$\mathcal{P}(y, X)$ the list in $\mathup{K}[X]$ 
obtained from $\mathcal{P}$ after specialization of
$Y_1, \dots, Y_m$ at $y_1, \dots, y_m$. 
Let $\tau$ be a sign condition on $\tmop{HElim} (P,
\mathcal{P})$. 
The lists
$\tmop{Feas}_{\rm sign} ( \mathcal{P}(y, X), Z(y)_{\mathup{R}})$ and 
$\tmop{Feas} ( \mathcal{P}(y, X),  Z(y)_{\mathup{C} \setminus \mathup{R}})$ and
their corresponding lists  
$c_{\rm sign} ( \mathcal{P}(y, X), Z(y)_{\mathup{R}})$,
and $c( \mathcal{P}(y, X),
Z(y)_{\mathup{C} \setminus \mathup{R}})$
of cardinals
are fixed as $y$ varies in $\tmop{Real}_{\rm sign}(\tau,
\mathup{R}^m)$.
\end{theorem}

\begin{proof}{Proof :}
For every $P, Q \in \mathup{K}[X]$, 
the signs of the  principal minors of $\tmop{Her}(P, Q)$ determine the rank and
signature of $\tmop{Her}( P, Q )$ (see  \cite[Chapter 9]{BPR}). 
Therefore, by Remark \ref{remark_used_znz} and 
the analogous result in \cite[Chapter 10]{BPR}, 
as $y$ varies in $\tmop{Real}_{\rm sign}(\tau,
\mathup{R}^m)$, all the calls to the Tarski-query blackbox and
invertibility-query blackbox done in Algorithm \texttt{Real-nonreal Sign Determination}
provide the same result; from this, we conclude that the output of 
this algorithm is the same. 
\end{proof}

\end{document}